\documentclass[a4paper,DIV8,10pt,openright]{scrartcl}
  
  \usepackage[applemac]{inputenc} 
  \usepackage[T1]{fontenc}
  \newcommand{\changefont}[3]{
  \fontfamily{#1} \fontseries{#2} \fontshape{#3} \selectfont}
  \changefont{ppl}{m}{n}

  \usepackage{amsmath}  
  \usepackage{amsthm}    
  \usepackage{amsfonts}  
   \usepackage{amssymb} 
   \usepackage{mathrsfs} 
  \usepackage{url}       
  \usepackage{enumitem} 
  \usepackage[headinclude]{scrpage2}  
  
  \usepackage{hyperref}
\usepackage{tikz}
\usetikzlibrary{calc}
\usetikzlibrary{matrix,arrows,decorations.pathmorphing}
  
  \usepackage[style=alphabetic, hyperref=true]{biblatex} 
  \bibliography{bibfile}
  
\usepackage{txfonts}

  \newcommand{\Z}{\ensuremath{\mathbb{Z}}}   
  \newcommand{\R}{\ensuremath{\mathbb{R}}}   
  \newcommand{\C}{\ensuremath{\mathbb{C}}}   
  \newcommand{\T}{\ensuremath{\mathbb{T}}}   

\DeclareMathOperator{\di}{div} 
  \DeclareMathOperator{\codim}{codim} 
  \DeclareMathOperator{\Hom}{Hom} 
  \DeclareMathOperator{\cone}{cone} 
  \DeclareMathOperator{\st}{star} 
   \DeclareMathOperator{\TZ}{TZ} 
      \DeclareMathOperator{\M}{M} 
   
  \newtheorem{satz}{Satz}[section]
  
  \newtheorem{proposition}[satz]{Proposition}
  
  \theoremstyle{definition}
  \newtheorem{definition}[satz]{Definition}

  \newtheorem{remark}[satz]{Remark}

    {\begin{proof}[Beweis]}
    {\end{proof}}
    {\begin{proof}[Sketch to the proof]}
    {\end{proof}}

  \newtheorem{example}[satz]{Example}

  \newtheorem{claim}[satz]{Claim}

\makeatletter

\renewcommand{\p@enumi}{}

\setkomafont{descriptionlabel}{\rmfamily}

  \makeatletter                    
    \def\blfootnote{\xdef\@thefnmark{}\@footnotetext}

  \title{Tropical intersection theory on $\R^n$}
  \author{Simon Flossmann}
\date{ }

\begin{document}
\maketitle
\thispagestyle{empty}
\begin{abstract}
\noindent In these notes we survey the tropical intersection theory on $\R^n$ by deriving the properties for tropical cycles from the corresponding properties in Chow cohomology. For this we review the stable intersection product introduced by Mikhalkin and the push forward of tropical cycles defined by Allermann and Rau. Furthermore we define a pull back for tropical cycles based on the pull back of Minkowski weights. This pull back commutes with the tropical intersection product and satisfies the projection formula. Our main result is to deduce the latter from the corresponding projection formula in Chow cohomology.
\end{abstract}

\section{Introduction}
In tropical geometry there is an intersection product for tropical cycles. The astonishing fact is that this product is well-defined as a tropical cycle in contrast to algebraic intersection theory or homology, where an equivalence relation is needed. An intersection product for tropical cycles has been introduced by Mikhalkin and Aller\-mann-Rau. Mikhalkin \cite{ap} uses a stable intersection product whereas Allermann and Rau use reduction to the diagonal and intersections with Cartier divisors as in \cite{allermann}. It is shown in \cite{katz} and \cite{rau} that both definitions agree. The goal of these notes is to deduce the usual properties of tropical cycles from the corresponding properties in Chow cohomology of toric varieties. This will lead to a new proof of the projection formula. These notes serve as back up for my talk at the Student Tropical Algebraic Geometry Seminar in Yale and are a summary of my master thesis \cite{simon}.

\par

\emph{The notes are organized as follows:} In the second section we review briefly intersection theory on toric varieties. In the third section we recall how Chow cohomology classes are relateted to Minkowski weights. In section 4 we recall the definition of tropical cycles and the stable intersection product for tropical cycles. In section 5 we review the push forward defined by \cite{allermann} and show in proposition \ref{prop:katz} that the push forward for tropical fans coincides with the induced push forward on the level of Chow cohomology induced by a projective toric morphism between the corresponding smooth and projective toric varieties, which was claimed in \cite{katz}. In section 6 we define the pull back also for tropical cycle based on the pull back of Minkowski weights as in \cite{fultonsturm}. Finally in the last section we show that this pull back commutes with the intersection product and is connected to the push forward by the projection formula.\par
 
\emph{Terminology:} Throughout these notes let $N$ and $N'$ denote free $\Z$-modules of finite ranks $r$ and $r'$. We write $N_{\R}:=N\otimes_{\Z} \R$ and $ N'_{\R}:= N' \otimes_{\Z} \R$.  Every integral, \R-affine polyhedron $\sigma$ of dimension $n$ in the $\R$-vector space $N_{\R}$ determines a $\R$-affine subspace $\mathbb{A}_{\sigma}$ with underlying vector space $\mathbb{L}_{\sigma}$ and lattice $N_{\sigma} := \mathbb{L}_{\sigma} \cap N$ in $\mathbb{L}_{\sigma}$. Given a polyhedral complex $\mathscr{C}$ of dimension $n$ in $N_{\R}$, we denote by $\mathscr{C}_n$ the subset of $n$-dimensional polyhedra in $\mathscr{C}$ and by $\mathscr{C}^{l}:= \mathscr{C}_{r-l}$ the subset of polyhedra in $\mathscr{C}$ of codimension $l$ in $N_{\R}$. We write $\tau \prec \sigma$ if $\tau $ is a face of $\sigma$. We say that a polyhedral complex $\mathscr{C}$ is of \emph{pure dimension} $n$ (resp. of \emph{pure codimension} $l$) if all polyhedra in $\mathscr{C}$ which are maximal with respect to $\prec$ lie in $\mathscr{C}_n$ (resp. $\mathscr{C}^{l}$). Given a polyhedral complex $\mathscr{C}$ of pure dimension $n$ and $m \le n$, we denote by $\mathscr{C}_{\le m}$ the polyhedral subcomplex of $\mathscr{C}$ of pure dimension $n$ given by all $\sigma \in \mathscr{C}$ with $\dim(\sigma) \le m$. And we set $\mathscr{C}^{\ge l}:= \mathscr{C}_{\le r -l}$ if $r-l \leq n$. An integral $\R$-affine polyhedron $\sigma$ in $N_{\R}$ which satisfies $\R_{\ge 0} \cdot \sigma = \sigma$ is called a \emph{rational polyhedral cone}. A rational polyhedral cone is called \emph{strictly convex} if it does not contain any line. A polyhedral complex consisting of strictly convex, rational, polyhedral cones is called a \emph{rational polyhedral fan}. We call a rational polyhedral fan \emph{regular} if every cone is \emph{regular}, i.e. its minimal generators form a part of a $\Z$-basis of $N$.\par
\emph{Acknowledgements:} I would like to thank Walter Gubler for suggesting this work to me and for numerous helpful discussions.

\section{Intersection theory of toric varieties}

In this section we briefly review the intersection theory on toric varieties. For more details see \cite{fultonsturm} and \cite{fulton}.

Following the notations of \cite{fultonsturm} and \cite{cox}, let $X_{\Sigma}$ be a toric variety  corresponding to a rational polyhedral fan $\Sigma$ in $N_\R$. Each cone $\tau \in \Sigma$ corresponds to an orbit $O(\tau)$ of the torus action, and the closure of $O(\tau)$ in $X_{\Sigma}$ is a closed torus-invariant subvariety of $X$ of the form $V(\tau)$. The dimension of $V( \tau)$ is given by
\begin{eqnarray*}
		\dim(V(\tau)) = r - \dim(\tau) = \codim(\tau).
\end{eqnarray*}
For each cone $\tau \in \Sigma^{k+1}$ we have the sublattices $N_{\tau}$ of $N$ and $M(\tau) := \tau^{\bot} \cap M$ of the dual lattice $M$ of $N$. Every nonzero $u \in M(\tau)$ determines a rational function $\chi^{u}$ on $V(\tau)$. The divisor of $\chi^{u}$ is given by
\begin{eqnarray}
		\di(\chi^{u}) = \sum_{\substack{\sigma \in \Sigma^{k}: \\ \tau \prec \sigma}} \langle u, \omega_{\sigma, \tau} \rangle \cdot V(\sigma), \label{eq:relation}
\end{eqnarray}
where $\omega_{\sigma, \tau}$ is a representative in $N_{\sigma}$ of the generator of the one-dimensional lattice $N_{\sigma} / N_{\tau}$ pointing in the direction of $\sigma$.\par
For an arbitrary variety $X$ we define the Chow group $A_{k}(X)$ as the group generated by the $k$-dimensional irreducible closed subvarieties of $X$, with relations generated by divisors of nonzero rational function on some $(k + 1)$-dimensional subvariety of $X$. In the case of toric varieties there is a nice presentation of the Chow groups  in terms of torus-invariant subvarieties and torus-invariant relations.

\begin{proposition}\label{prop:chowgroups}
 \begin{enumerate}
 \item The Chow group $A_{k}(X)$ of a toric variety $X_{\Sigma}$ is generated by the classes $[V(\sigma)]$, where $\sigma$ runs over all cones of codimension $k$ in the rational polyhedral fan $\Sigma$.
 \item The group of relations on these generators is generated by all relations \eqref{eq:relation}, where $\tau$
runs over cones of codimension $k + 1$ in $\Sigma$, and $u$ runs over a generating set of $M(\tau)$.
 \end{enumerate}
\end{proposition}
For a proof see \cite[Proposition 1.1]{fultonsturm}.

\section{Chow cohomology of toric varieties and Minkowski weights}

For a complete toric variety $X_{\Sigma}$ it is known \cite[Proposition 1.4]{fultonsturm}, that the Chow cohomology group $A^{k}(X_{\Sigma})$ is canonically isomorphic to $\Hom(A_{k}(X_{\Sigma}), \Z)$. Thus we can describe $A^{k}(X_{\Sigma})$, following \cite{fultonsturm}, as follows: 

\begin{definition}
Let $\Sigma$ be a complete rational polyhedral fan in $N_{\R}$. An integer-valued function 
\begin{eqnarray*}
c \colon \Sigma^{k} &\to &\Z \\
 \sigma & \mapsto & c_{\sigma}
\end{eqnarray*}
is called a \emph{Minkowski weight} of codimension $k$, if it satisfies the \emph{balancing condition}, that is, for every $\tau \in \Sigma^{k+1}$,
\begin{eqnarray*}
		\sum_{\substack{\sigma \in \Sigma^{k}: \\ \tau \prec \sigma}} c_{\sigma} \cdot \omega_{\sigma, \tau} \in N_{\tau},
	\end{eqnarray*}
where $\omega_{\sigma, \tau}$ is a representative in $N_{\sigma}$ of the generator of the one-dimensional lattice $N_{\sigma} / N_{\tau}$ pointing in the direction of $\sigma$.
\end{definition}

\begin{proposition}\label{prop:isomorphism}
The Chow cohomology group $A^{k}(X_{\Sigma})$ of a complete toric variety $X_{\Sigma}$ is canonically isomorphic to the group of Minkowski weights $\M^{k}(\Sigma)$ of codimension $k$ on $\Sigma$.
\end{proposition}

This proposition is an immediate consequence of the description of $A^{k}(X_{\Sigma})$ in Proposition \ref{prop:chowgroups} and the isomorphism 
\begin{eqnarray*}
A^{k}(X_{\Sigma})& \to & \Hom(A^{k}(X_{\Sigma}), \Z) \\  c &\mapsto &\big(\alpha \to \deg(c \frown \alpha)\big).
\end{eqnarray*}

Identifying Chow cohomology classes with Minkowski weights we can describe the cup product of Chow cohomology classes combinatorial by the \emph{fan displacement rule}:

\begin{proposition}\label{prop:fandisplacementrule}
Let $c \in \M^{k}(\Sigma)$ and $d \in \M^{l}(\Sigma)$ then their \emph{product} is given by
\begin{eqnarray*}
\smallsmile \, \colon \M^{k+l}(\Sigma) & \to & \Z\\
c \smallsmile d & \mapsto &(c \smallsmile d)(\tau) := \sum_{\substack{\sigma \in \Sigma^{(k)}: \\ \tau \prec \sigma} } \, \sum_{\substack{ \sigma' \in \Sigma^{(l)}: \\ \tau \prec  \sigma'} } m^{\tau}_{\sigma, \sigma'} \cdot c_{\sigma}\cdot d_{\sigma'},
\end{eqnarray*}where the coefficient $m^{\tau}_{\sigma, \sigma'}$ is not uniquely defined but depend on the choice of a generic vector $v \in N_{\R}$ (cf. \cite[Sec. 3]{fultonsturm} for the precise definition). For such a vector we have	
$$
	m^{\tau}_{\sigma, \sigma'}  := \begin{cases} [N: N_{\sigma} + N_{\sigma'}], & \text{ if } \sigma \cap (\sigma' +v) \not = \emptyset \text{ and }  \\
0, & \text{ otherwise. } \end{cases}
$$
\end{proposition} 
For a proof see \cite[Proposition 3.1 (a)]{fultonsturm}. This product is associative, commutative, with identity and turns the direct sum of $\Z$-moduls
\begin{eqnarray*}
\M^{*}(\Sigma) := \bigoplus_{k= 0, \ldots ,r} \M^{k}(\Sigma)
\end{eqnarray*}
into a graded, associative and commutative $\Z$-algebra with identity.

\section{Stable intersection product for tropical cycle}

Before we review the definition of the intersection product for tropical cycle we will define tropical cycles, following \cite{allermann}, \cite{intersection} and \cite{gubler} in the real vector space $N_\R$.
\begin{definition}
\begin{enumerate}
		\item A integral $\R$-affine polyhedral polyhedral complex $\mathscr{C}$ of pure dimension $n$ is called \emph{weighted}, if it is endowed with a \emph{integral weight function} m, which maps every $\sigma \in \mathscr{C}_n$ to a number $m_{\sigma} \in \Z$. The \emph{support} of $\mathscr{C}$ is defined as
		\begin{eqnarray*}
			|\mathscr{C}| := \bigcup_{\substack{\sigma \in \mathscr{C}_n :\\ m_{\sigma} \not = 0}}  \sigma.
		\end{eqnarray*} And called \emph{complete} if $|\mathscr{C}| = N_{\R}$. 
		\item Let $(\mathscr{C},m)$ be a weighted, integral $\R$-affine polyhedral complex of pure dimension $n$ in $N_{\R}$. For each codimension one face $\tau$ of a polyhedron $\sigma \in \mathscr{C}$ we choose a representative $\omega_{\sigma, \tau} \in N_{\sigma}$ of the generator of the one-dimensional lattice $N_{\sigma} / N_{\tau}$ pointing in the direction of $\sigma$. We call this vector a \emph{normal vector} of $\sigma$ in direction $\tau$. Then we say that the weighted polyhedral complex $\mathscr{C}$ satisfies the \emph{balancing condition}, if for every $\tau \in \mathscr{C}_{n-1}$ we have
	\begin{eqnarray*}
		\sum_{\substack{\sigma \in \mathscr{C}_{n}: \\ \tau \prec \sigma}} m_{\sigma} \cdot \omega_{\sigma, \tau} \in N_{\tau}.
	\end{eqnarray*}
	
	\item Let $C := (\mathscr{C}, m)$ be a weighted, integral $\R$-affine polyhedral complex of dimension $n$, which satisfies the balancing condition. Given an integral, $\R$-affine subdivision $\mathscr{C}'$ of $\mathscr{C}$ there is a unique weight function $m'$ such that $(\mathscr{C}',m')$ again satisfies the balancing condition and $m_{\sigma}|_{\sigma'}= m'_{\sigma'}$ holds for all $\sigma' \in \mathscr{C}'_n $ and $\sigma \in \mathscr{C}_n$ such that  $\sigma' \subseteq \sigma$.
 In the following we identify two weighted polyhedral complexes, which satisfy the balancing condition, $(\mathscr{C}, m)$ and $(\mathscr{C}', m')$ of dimension $n$, if and only if, if there exists a common subdivision $\mathscr{D}$ of $\mathscr{C}$ and $\mathscr{C}'$ such that the weight functions induced by $m$ and $m'$ on all polyhedra in $\mathscr{D}_n$ agree. 
A \emph{tropical cycle} $C := (\mathscr{C}, m)$ of dimension $n$ in $N_{\R}$ is an equivalence class of an integral, $\R$-affine and weighted polyhedral complex $\mathscr{C}$ in $N_{\R}$ of pure dimension $n$, which satisfies the balancing condition. (See Figure \ref{fig:einleitung} for an example.) If a tropical cycle $C$ in $N_{\R}$ is defined by a weighted integral, $\R$-affine polyhedral complex $ (\mathscr{C}, m)$, then we call $\mathscr{C}$ a \emph{polyhedral complex of definition for the tropical cycle} $C$. A tropical cycle $C := (\mathscr{C}, m)$, whose underlying polyhedral complex consists of strictly convex rational polyhedral cones in $N_{\R}$ is called \emph{tropical fan}.
	
\item The set of tropical cycle of pure dimension $n$ in $N_{\R}$ defines an abelian group $\TZ_n(N_{\R})$, where the group law is given by the addition of weight functions on a common refinement of the integral $\R$-affine polyhedral complex. We denote by $\TZ^{k}(N_{\R}) := \TZ_{r-k}(N_{\R})$ the \emph{group of tropical cycles of codimension k}.

\item An operation on tropical cycles, which mirrors the local structure, is that of taking stars. Let $C$ a tropical cycle with polyhedral complex of definition $\mathscr{C}$. Let $\tau \in \mathscr{C}$. We denote by $\overline{\sigma}$ the cones in the images of the polyhedra of $\mathscr{C}$ containing $\tau$ under the canonical projection map $N_{\R} \to N_{\R}/ \mathbb{L}_{\tau}$. And define the \emph{star} of $C$ at $\tau$ to be the rational polyhedral fan in $N_{\R}/ \mathbb{L}_{\tau}$ given by
$$
\st_{\mathscr{C}}(\tau):= \big\{  \overline{\sigma} \,|\, \sigma \in \mathscr{C} \text{ with } \tau \prec \sigma \big\}.
$$
If we equip the $\st_{\mathscr{C}}(\tau)$ with the induced weights and since the projection respects lattice normal vectors, the star of $\mathscr{C}$ at $\tau$ will be a tropical fan, which we will denote by $\st_{C}(\tau)$. We also note that the projection map preserves the codimension of the polyhedra.
\end{enumerate}
\end{definition}

\begin{figure}[h]

\begin{center}
    \begin{tikzpicture}
    \begin{scope}[color = gray ,text width=2em,scale=0.6,font={\scriptsize},line width=1pt]
    \draw[]  (0,0) -- (0,-2);
    \draw[]  (0,0) -- (2,0);
    \draw[]  (0,0) -- (0,-2);
    \draw[]  (2,0) -- (2,-2);
    \draw[]  (2,0) -- (4,2);
    \draw[]  (0,0) -- (-2,2);
    \draw[]  (-2,2) -- (-4,2);
    \draw[]  (-2,2) -- (-2,4);
    \draw[]  (-2,4) -- (-4,4);
    \draw[]  (-2,4) -- (0,6);
    
    \end{scope}

    \begin{scope}[color = black ,text width=2em,scale=0.6,font={\scriptsize},line width=1pt]
    \draw[->]  (0,0) -- (0,-0.7);
    \draw[->]  (0,0) -- (-0.5,0.5);
    \draw[->]  (0,0) -- (0.7,0); 
    
     \node[left] at (0,-0.7) {$\binom{0}{-1}$};
    \node[below] at (0.7,0) {$\tbinom{1}{0}$};
    \node[right] at (-1,1) {$\tbinom{-1}{1}$};
 
    \draw[->]  (2,0) -- (2.5,0.5);
    \draw[->]  (2,0) -- (2,-0.7);
    \draw[->]  (2,0) -- (1.3,0); 
    
     \node[below] at (1.5,0) {$\tbinom{-1}{0}$};
    \node[left] at (3,1) {$\tbinom{1}{1}$};
    \node[right] at (2,-0.7) {$\tbinom{0}{-1}$};
 
    \draw[->]  (-2,2) -- (-1.5,1.5);
    \draw[->]  (-2,2) -- (-2.7,2);
    \draw[->]  (-2,2) -- (-2,2.7); 
    
     \node[right] at (-2,2.5) {$\binom{0}{1}$};
    \node[below] at (-2.7,2) {$\tbinom{-1}{0}$};
    \node[left] at (-1,1) {$\tbinom{1}{-1}$};
    
     \draw[->]  (-2,4) -- (-2.7,4);
    \draw[->]  (-2,4) -- (-2,3.3);
    \draw[->]  (-2,4) -- (-1.5,4.5); 
    
     \node[right] at (-2,3.5) {$\binom{0}{-1}$};
    \node[above] at (-2.7,4) {$\tbinom{-1}{0}$};
    \node[above] at (-1.5,4.5) {$\tbinom{1}{1}$};
    \end{scope}

   \end{tikzpicture} 
 \end{center}

  \caption{A tropical cycle in $\R^{2}$.}
  \label{fig:einleitung}
\end{figure}
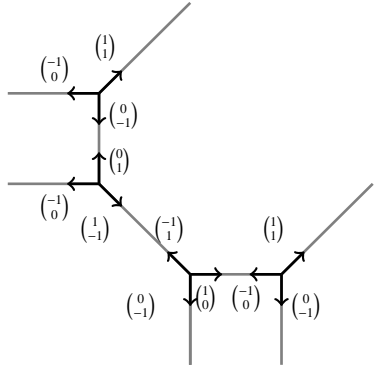

Now we can define the intersection product for tropical cycles following \cite{ap}:
\begin{definition}\label{def:intersectionproduct}
Let $C_1$ and $C_2$ be tropical cycle of codimension $k_1$ and $k_2$ in $N_{\R}$. Then there exists a natural pairing 
	\begin{eqnarray*}
			.  \, \colon  \TZ^{k_1}(N_{\R}) \times \TZ^{k_2}(N_{\R}) & \to & \TZ^{k_1+k_2}(N_{\R})\\
			(C_{1}, C_{2}) & \mapsto & C_{1}.C_{2},	
	\end{eqnarray*}
which is called the \emph{stable intersection product for tropical cycles}. This pairing is bilinear, associative, commutative and respects supports in the sense that $|C_1 . C_2| \subseteq |C_1| \cap |C_2|$. It induces a graded, associative and commutative $\Z$-algebra with identity 
$$\TZ^{*}(N_{\R}):= \bigoplus_{k=0, \ldots, r} \TZ^{k}(N_{\R}),$$
 where the scalar multiplication acts on the multiplicities. Concretely this intersection pairing is given by the \emph{local fan displacement rule} adapted from \cite{fultonsturm}, which we explain in the following: We choose a common polyhedral complex of definition $\mathscr{C}$ for the tropical cycles  $C_1$ and $C_2$ and write $C_i := (\mathscr{C},m_i)$ (i = 1,2) for the weight functions $m_i := (m_{i,\sigma})_{\sigma \in \mathscr{C}^{k_i}}$. Now let $\mathscr{D}$ denote the polyhedral subcomplex of $\mathscr{C}$, which is generated by $\mathscr{C}^{k_1+k_2}$.
In order to define a well-defined weight function on $\mathscr{D}$, we need to make sure that the tropical cycles intersect transversal. For this we choose a generic vector $v \in N_{\R}$ (cf. \cite[Sec. 3]{fultonsturm} for the precise definition) and a small $\varepsilon > 0$ and now equip $\mathscr{D}$ with the integral weight function $m := (m_{\tau})_{\tau \in \mathscr{D}^{k_1+k_2}}$ given by
$$m_{\tau} := \sum_{\substack{(\sigma_1, \sigma_2) \in \mathscr{C}^{k_1} \times \mathscr{C}^{k_2} : \\ \tau = \sigma_1 \cap \sigma_2 }} m^{\tau}_{\sigma_1, \sigma_2}\cdot m_{1,\sigma_1}  \cdot m_{2,\sigma_2},
	$$where
	\begin{eqnarray*}
	m^{\tau}_{\sigma_1, \sigma_2}  := \begin{cases}   [N : N_{\sigma_1} + N_{\sigma_2}] , & \text{ if } \sigma_1 \cap (\sigma_2 + \varepsilon \cdot v) \not = \emptyset  \text{ and } \\
0, & \text{ otherwise.}\end{cases}
	\end{eqnarray*}
	Then $D:=(\mathscr{D},m)$ is a tropical cycle of codimension $k_1 +k_2$, whose construction is independent of the choice of the generic vector $v$ and a sufficiently small $\varepsilon \in \R_{>0}$ and $\mathscr{D}$ is a polyhedral complex of definition for the stable intersection product $C_1 . C_2$.
\end{definition}
\begin{claim}Definition \ref{def:intersectionproduct} is well-defined and satisfies the claimed properties.
\end{claim}

\begin{proof} 
We will follow the proof given in \cite{intersection} to a great extend. Since the construction of the stable intersection product is local, it is enough to proof the well-definedness in a neighborhood of $\tau\in \mathscr{D}$. After dividing out the vector space $\mathbb{L}_{\tau}$ and using the star construction we get tropical fans $\st_{C_{i}}(\tau)$ in the quotient vector space $N_{\R}/\mathbb{L}_{\tau}$. Choosing a rational polyhedral fan of definition $\st_{\mathscr{C}_i}(\tau)$ for $\st_{C_{i}}(\tau)$ we can assume that $\st_{\mathscr{C}_i}(\tau)$ are Minkowski weights in a complete fan $\Sigma$  (cf. \cite[Lem. 1.1.19]{fran} for details). We know by proposition \ref{prop:isomorphism} that the Minkowski weights $\st_{\mathscr{C}_i}(\tau)$ correspond to cocycles $c_{i} $ in the Chow cohomology groups $A^{k_i}(X_{\Sigma})$. Now we can compute the product of these cocycles in terms of the fan displacement rule (cf. proposition \ref{prop:fandisplacementrule}), which is again a Minkowski weight and so also a well-defined tropical cycle. Now it remains to show that the resulting Minkowski weight in $N_{\R}/\mathbb{L}_{\tau}$ coincide with our given weight function $m$:\par
After choosing a generic vector $\overline{v} \in \mathbb{L}_{\sigma}/\mathbb{L}_{\tau}$ we get
\begin{eqnarray*}
 (c_{1} \smallsmile c_{2})(\overline{\tau}) & = &	\sum_{\substack{(\overline{\sigma_1}, \overline{\sigma_2}) \in \st_{\mathscr{C}_1}(\tau)^{k_1} \times \st_{\mathscr{C}_2}(\tau)^{k_2} : \\ \overline{\tau} = \overline{\sigma_1} \cap \overline{\sigma_2} }} m^{\overline{\tau}}_{\overline{\sigma_1}, \overline{\sigma_2}} \cdot c_{1,\overline{\sigma_1}}  \cdot c_{2,\overline{\sigma_2}}\\
&=&\sum_{\substack{(\overline{\sigma_1}, \overline{\sigma_2}) \in \st_{\mathscr{C}_1}(\tau)^{k_1} \times \st_{\mathscr{C}_2}(\tau)^{k_2} : \\ \overline{\tau} = \overline{\sigma_1} \cap \overline{\sigma_2} }}    \Big[N/ N_{\tau}: \big(N/N_{\tau}\big)_{\overline{\sigma_{1}}}+ \big(N/N_{\tau}\big)_{\overline{\sigma_{2}}}\Big]  \cdot c_{1,\overline{\sigma_{1}}}\cdot c_{2,\overline{\sigma_{2}}}\\
	 &=&  \sum_{\substack{(\sigma_1, \sigma_2) \in \mathscr{C}_1^{k_1} \times \mathscr{C}_2^{k_2} : \\ \tau = \sigma_1 \cap \sigma_2 }}  [N: N_{\sigma}+ N_{\sigma'}] \cdot m_{1,\sigma_{1}} \cdot m_{2,\sigma_{2}}\\
	 & = & m_{\tau},
 \end{eqnarray*}
	where we used that with $\overline{\sigma_{1}} \cap (\overline{\sigma_{2}} + \overline{v}) \not = \emptyset $ it follows that $\sigma_{1} \cap (\sigma_{2} + v) \not = \emptyset $ for a preimage $v$ of $\overline{v}$ and
	\begin{eqnarray}
	\Big[N/ N_{\tau}: \big(N/N_{\tau}\big)_{\overline{\sigma_{1}}}+ \big(N/N_{\tau}\big)_{\overline{\sigma_{2}}}\Big]  &=&\Big|\big(N/ N_{\tau}\big)\Big/ \Big(\big(N/N_{\tau}\big)_{\overline{\sigma_{1}}}+ \big(N/N_{\tau}\big)_{\overline{\sigma_{2}}}\Big)\Big| \notag \\ 
	&=& \Big|(N/ N_{\tau})\Big/ \big(N_{\sigma_{1}}+ N_{\sigma_{2}}\big)\big/N_{\tau}\Big| \notag \\
	&=& \big|N / (N_{\sigma_{1}}+ N_{\sigma_{2}})\big| \label{eq:well:1},\\
	&=&\big[N: N_{\sigma_{1}}+ N_{\sigma_{2}}\big]\label{eq:well:2},
	\end{eqnarray}
where \eqref{eq:well:1} and \eqref{eq:well:2} follows with the second and third isomorphism theorem. The last step is to see that the above is independend of the choosen fan of definition. For this let $\Sigma'$ be a refinement of a complete fan $\Sigma$ in $N_{\R}$. Let 
$$
f \colon X_{\Sigma'} \to X_{\Sigma}
$$
be the induced morphism of the corresponding toric varieties. This morphism induces the homomorphism 
$$
f^{*} \colon A^{k}(X_{\Sigma}) \to A^{k}(X_{\Sigma'}),
$$of the corresponding Chow cohomology groups, see \cite[p.324]{fulton}. Now we will show that $f^{*}$ is injecitve. To proof this let $c \in A^{k}(X_{\Sigma})$ be a cocycle and $[X_{\Sigma'}] \in A^{k}(X_{\Sigma'})$ be the fundamental cocycle. Using that $f$ is birational the projection formula \cite[p. 324]{fulton} yields
$$
f_{*}\big(f^{*} c \frown [X_{\Sigma'}]\big) = c \frown [X_{\Sigma}]
$$
and so the injecitvity of $f^{*}$ follows, since
$$
f_{*}(f^{*} c) = c.
$$ The injectivity of $f^{*}$ shows that building the intersection product of tropical cycles does not depend on the chosen fan of definition.\par 
Finally the remaining claimed properties of the stable intersection product follow from the properties of the Chow cohomology groups (cf. \cite[Ch. 17]{fulton}).
\end{proof}

\begin{remark}\label{re:verfeinerung}
In the following we will refer to the process of interpreting tropical fans as Minkow\-ski weights and as Chow cohomology classes as \emph{standard procedure}. 
\end{remark}

\begin{example}
We will illustrate the stable intersection product by computing the stable intersection of two tropical curves. The polyhedral complexes of definition of the given curves $C_1$ and $C_2$ are pictured in figure  \ref{fig:intersection} on the left side. The picture on the left side should be read that the most parts of the supports of the curves lie on top of each other. Furthermore we assume that the weight on each ray is 1 except where indicated. The next step is to choose a generic vector, which is the same as to shift the curve $\mathscr{C}_2$ a little bit such that the intersection is now transversal. This is pictured in figure \ref{fig:intersection} in the middle. Now we can read of the weights of the intersection cycle $\mathscr{D}$ with the fan displacement rule from the picture in the middle. In the picture on the right hand side we see the resulting intersection cycle $D$ of the two given curves $C_1$ and $C_2$.

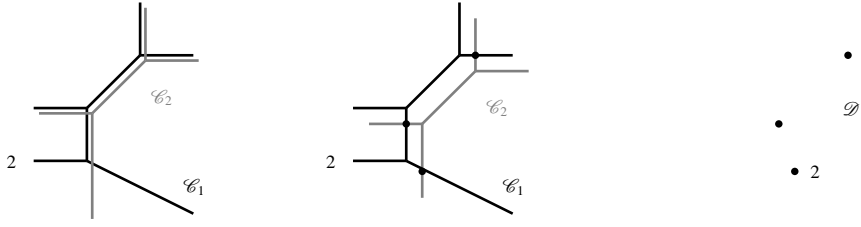
\begin{figure}[h]
  \centering

    \begin{tikzpicture} 
       
     \begin{scope}[color = black ,text width=2em,scale=0.7,font={\scriptsize},line width=1pt]
    \draw[]  (0,0) -- (2,-1);
    \draw[]  (0,0) -- (0,1);
    \draw[]  (0,0) -- (-1,0); 
    \draw[]  (0,1) -- (-1,1); 
    \draw[]  (0,1) -- (1,2); 
    \draw[]  (1,2) -- (2,2); 
    \draw[]  (1,2) -- (1,3); 
     \node[left] at (3,-0.5) {$\mathscr{C}_1$};
      \node[below] at (-1,0.3) {$2$};
       \end{scope}
     
     \begin{scope}[color = gray ,text width=2em,scale=0.7,font={\scriptsize},line width=1pt]
    \draw[]  (1.1,1.9) -- (2.1,1.9); 
    \draw[]  (1.1,1.9) -- (1.1,2.9); 
    \draw[]  (0.1,0.9) -- (1.1,1.9); 
     \draw[]  (0.1,0.9) -- (-0.9,0.9); 
      \draw[]  (0.1,0.9) -- (0.1,-1.1); 
    \node[right] at (1,1.2) {$\mathscr{C}_2$};
    \end{scope}
    
    
       \begin{scope}[color = gray ,text width=2em,scale=0.7,font={\scriptsize},line width=1pt]
    \draw[]  (7.3,1.7) -- (8.3,1.7); 
    \draw[]  (7.3,1.7) -- (7.3,2.7); 
    \draw[]  (6.3,0.7) -- (7.3,1.7); 
     \draw[]  (6.3,0.7) -- (5.3,0.7); 
      \draw[]  (6.3,0.7) -- (6.3,-0.7); 
    \node[right] at (7.3,1) {$\mathscr{C}_2$};
    \end{scope}

         \begin{scope}[color = black ,text width=2em,scale=0.7,font={\scriptsize},line width=1pt]
    \draw[]  (6,0) -- (8,-1);
    \draw[]  (6,0) -- (6,1);
    \draw[]  (6,0) -- (5,0); 
    \draw[]  (6,1) -- (5,1); 
    \draw[]  (6,1) -- (7,2); 
    \draw[]  (7,2) -- (8,2); 
    \draw[]  (7,2) -- (7,3); 
     \node[left] at (9,-0.5) {$\mathscr{C}_1$};
      \node[below] at (5,0.3) {$2$};
      \fill (7.3,2) circle (0.07);
      \fill (6,0.7) circle (0.07);
      \fill (6.3,-0.2) circle (0.07);
       \end{scope}
       
    \begin{scope}[color = black ,text width=2em,scale=0.7,font={\scriptsize},line width=1pt]
    \fill (14.3,2) circle (0.07);
      \fill (13,0.7) circle (0.07);
      \fill (13.3,-0.2) circle (0.07);
       \node[right] at (13.4,-0.2) {$2$};
       \node[right] at (14,1) {$\mathscr{D} $};
     \end{scope}

   \end{tikzpicture} 
  \caption{Stable tropical intersection of two tropical curves.}
  \label{fig:intersection}
\end{figure}
\end{example}

\section{Push forward of tropical cycle}

Following \cite{allermann} we define the push forward along an integral $\R$-affine map.
\begin{remark}\label{re:push}
Let $F \colon N_{\R} \to N'_{\R}$ be an integral, $\R$-affine map. There is a natural \emph{push forward morphism}
\begin{eqnarray*}
F_{*} \colon \TZ_n (N_{\R}) & \to & \TZ_n (N'_{\R})\\
C & \mapsto & F_{*}C,
\end{eqnarray*}
which satisfies $|F_{*}C' | \subseteq F(|C|)$. It is given as follows: Let $C$ in $\TZ_n (N_{\R})$ be a tropical cycle. After a suitable refinement we can assume that 
\begin{eqnarray*}\label{eq:push}
F_{*}\mathscr{C}  := \big \{F(\tau) \big| \, \exists \, \sigma \in \mathscr{C}_n \text{ such that } \tau \prec \sigma \text{ and } F|_{\sigma} \text{ is injective} \big\}
\end{eqnarray*}
 is a polyhedral  complex in $N'_{\R}$. We equip $F_{*}\mathscr{C} $ with the weight function 
 $$
 m'_{\sigma'} := \sum_{\substack{\sigma \in \mathscr{C}_n:\\ F(\sigma)= \sigma'}} m_{\sigma}  \cdot \big[N'_{\sigma'}: \mathbb{L}_{F}(N_{\sigma})\big],
 $$ for $\sigma' \in F_{*}\mathscr{C}_n$, where $\mathbb{L}_{F}$ denotes the linear morphism defined by the affine morphism $F$. The weighted, integral and $\R$-affine polyhedral complex $F_{*}C := (F_{*}\mathscr{C}, m')$ in $N'_{\R}$ of pure dimension $n$ is called the \emph{direct image of $C$ under $F$}.
One verifies that $F_{*}C$ is again a tropical cycle of dimension $n$. The formation of $F_{*}$ induces a functorial homomorphism of the tropical cycle groups. For further details see \cite[Sec. 7]{allermann}.
\end{remark}

We illustrate the push forward of a tropical cycle.

\begin{example}\label{ex:projection}
For the projection map to the first coordinate $ \Z	^{2} \to \Z$ let
\begin{eqnarray*}
F \colon \R^{2} & \to & \R \\
(x,y) & \mapsto & x,
\end{eqnarray*}
be the induces map of vector spaces. We will calculate the push forward $F_{*}C$ of the tropical curve $C$ given in figure \ref{fig:pushforward}. First we remark that the given polyhedral complexes of definition $\mathscr{C}$ and $\mathscr{D}$ in figure \ref{fig:pushforward} are compatible with respect to $F$. Furthermore we assume that the weight  on each ray is 1 except where indicated.
\begin{figure}[h]
  \centering

    \begin{tikzpicture}

     \begin{scope}[color = black ,text width=2em,scale=0.7,font={\scriptsize},line width=1pt]
    \draw[]  (0,0) -- (2,-1);
    \draw[]  (0,0) -- (0,1);
    \draw[]  (0,0) -- (-1,0); 
    \draw[]  (0,1) -- (-1,1); 
    \draw[]  (0,1) -- (1,2); 
    \draw[]  (1,2) -- (2,2); 
    \draw[]  (1,2) -- (1,3); 
    
    \draw[]  (7,1) -- (10.5,1); 
    
    \fill (0,0) circle (0.07);
     \fill (1,2) circle (0.07);
     \fill (0.9,-0.45) circle (0.07);
    
    \fill (8,1) circle (0.07);
    \fill (9,1) circle (0.07);

    \draw[->]  (4,1) -- (5,1); 
    \node[above] at (4.9,1) {$F$};
    
    \node[below] at (0,0) {$2$};
    \node[right] at (1.2,-0.4) {$\sigma_{2}$};
     \node[above] at (1.9,2) {$\sigma_{1}$};
     \node[left] at (1,2) {$\mathscr{C}$};
    \node[above] at (10,1) {$\sigma'$};
    \node[right] at (11,1) {$\mathscr{D}$};

    \end{scope}
    
   \end{tikzpicture} 
  \caption{Push forward of a tropical curve.}
  \label{fig:pushforward}
\end{figure}
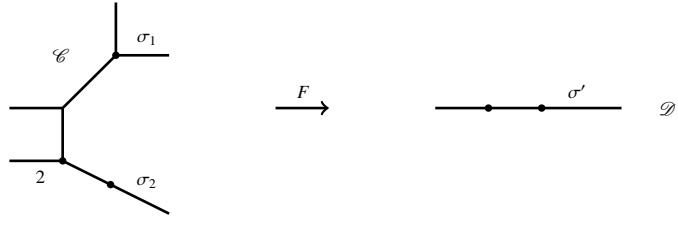
Now we get $F_{*} \mathscr{C} = \mathscr{D}$. And the weight on every facet is $3$, since for example we have
\begin{eqnarray*}
m'_{\sigma'} &=&  |\Z / \Z| \cdot m_{\sigma_{2}} + |\Z / 2 \Z| \cdot m_{\sigma_{1}}  \\
&= & 1+2 = 3.
\end{eqnarray*}
\end{example}

Now we proof the claim, which was made in \cite[Sec. 4]{katz}: 

\begin{proposition}\label{prop:katz}
Let $F \colon N_{\R} \to N'_{\R}$ be an integral linear map and $C$ a tropical fan in $N_{\R}$. Then there exists a projective toric morphism between smooth and projective toric varieties  $f \colon X_{\Sigma} \to X_{\Sigma'}$ such that the push forward in the Chow cohomology along this morphism coincides with the push forward of tropical fans of remark \ref{re:push}.
\end{proposition}

\begin{proof}
Let $C := (\mathscr{C}, m)$ be a tropical fan in $N_{\R}$ of dimension $n$. After a suitable refinement we may assume that $\mathscr{C}$ is a polyhedral complex of definition in a complete fan $\Sigma$. Furthermore we may assume that $\Sigma'$ is a complete fan in $N'_{\R}$ with $F(\sigma) \in \Sigma'$. Applying \cite[Thm. 6.1.18]{cox} we can also assume that the corresponding toric varieties $X_{\Sigma}$ and $X_{\Sigma'}$ are smooth and projective. Since $X_{\Sigma}$ and $X_{\Sigma'}$ are complete the induced morphism of toric varieties 
$$
f \colon X_{\Sigma} \to X_{\Sigma'}
$$
is proper. And since a proper morphism between projective varieties is projective this morphism is a local complete intersection \cite[Sec. 6.6]{fulton}. For a l.c.i. the push forward and pull back of cycles is well-defined \cite[Ch. 17]{fulton}.
With the standard procedure we can identify $C$ with a cocycle $c$ in the Chow cohomology group $A^{r-n}(X_{\Sigma})$.\par
At first we will calculate the push forward $f_{*}c$ in $A^{r'-n}(X_{\Sigma'})$: From the properties of the Chow groups we verify that the cocycle $c$ induces via
 \begin{eqnarray*}
c_f^{(l)} \colon A_{l}(X_{\Sigma'}) & \to & A_{l-(r-n)}(X_{\Sigma})\\
z' & \mapsto & c \smallfrown  f^{*}(z')
\end{eqnarray*}
a bivariante class (cf. \cite[Ch. 17]{fulton}) in $A^{r-n}(X_{\Sigma} \overset{f}{\to} X_{\Sigma'})$. As in \cite[17.2 $(P_2)$]{fulton} we define the push forward as
\begin{eqnarray*}
f_{*} \colon A^{r-n}(X_{\Sigma} \overset{f}{\to} X_{\Sigma'})& \to & A^{r-n}(X_{\Sigma} \overset{\text{id}}{\to} X_{\Sigma})\\
f_{*} c \smallfrown z' & \mapsto & f_{*} (c \smallfrown f^{*}(z')).
\end{eqnarray*}\par
In the next step we show that this push forward coincides with the push forward of $f$ in the case of Chow groups as in \cite[Sec. 1.4]{fulton}. Choose $z' := [X_{\Sigma}]$ to be the fundamental cycle, then by Poincar\'{e}-duality \cite[Cor. 17.4]{fulton} we get
\begin{eqnarray*}
	 f_{*}c \smallfrown X_{\Sigma'} &=& f_{*}(c \smallfrown f^{*} X_{\Sigma'}) \\
	 & = & f_{*}(c \smallfrown X_{\Sigma} ) \in A_{n}(X_{\Sigma'}).
\end{eqnarray*}
This shows that the push forward of Chow cohomology commutes with the induced isomorphism by the Poincar\'{e}-duality and so the both push forwards agree.\par
After identifying cocyles with tropical fans we need to show that $f_{*}c$ and $F_{*}\mathscr{C}$ coincide as tropical fans. Since tropical fans are completely determined by their weight function it is enough to show that those agree. Having closely examine how cycles define tropical fans in \cite[Thm. 1.1.15]{rau}, we see that is enough to show, that for $\sigma' \in \Sigma'_n$, the following holds
\begin{eqnarray*}
m'_{\sigma'} &=& \deg\big( f_{*}c \smallfrown [V(\sigma')] \big),
\end{eqnarray*} 
where $[V(\sigma')]$ is the chow class in $A_{n-r}(X_{\Sigma'})$ (cf. \cite[Sec. 12.5]{cox}) of the torus invariant closed subvariety $V(\sigma')$ corresponding to $\sigma'$ in $X_{\Sigma'}$\cite[Thm. 3.2.6]{cox}. Using the definition of $f_{*}$ and that  the degree map commutes with the push forward \cite[Def. 1.4]{fulton} the right hand side simplifies to 
\begin{eqnarray}\label{eq:deg}
\deg\big( f_{*}c \smallfrown [V(\sigma')] \big) &=&  \deg\bigg( f_{*}\big(c \smallfrown f^{*}([V(\sigma')])\big) \bigg) \notag \\
&=& \deg\bigg( c \smallfrown f^{*}\big([V(\sigma')]\big) \bigg).
\end{eqnarray} \par
Now we describe the occurring pull back: Since $f$ is proper and using the definition of $F_{*}$, we may assume that $V(\sigma)$ maps onto $V(\sigma')$ for $\sigma' \in \Sigma'_n$ and the pull back
 is given by
\begin{eqnarray}\label{eq:pullback}
		f^{*}([V(\sigma')]) = \sum_{\substack{\sigma \in \Sigma_{n}:\\ F(\sigma) = \sigma'} } b_{\sigma} \cdot [V(\sigma)].	
\end{eqnarray}

It remains to calculate the coefficients $b_{\sigma}$: Since $\Sigma $ and $\Sigma'$ are regular, we may assume that $\sigma$ and $\sigma'$ will be generated by the first $n$ base vectors $e_1, \ldots, e_n$ of $N = \Z^r$ and $e'_1, \ldots, e'_n$ of $N'=\Z^{r'}$. From $F(\sigma) = \sigma'$ we deduce that the the induced linear map is given by  
\begin{eqnarray*}
\mathbb{L}_{F}\colon N_{\sigma} &\to& N'_{\sigma}\\
e_i & \mapsto & b_i \cdot e'_i
\end{eqnarray*} 
for $i = 1, \ldots , n$ and $b_i \in \Z$. Since $F$ is injective on $\sigma$, we can interpret $\mathbb{L}_{F}(N_{\sigma})$ as a subgroup of $N'_{\sigma}$ and obtain 
\begin{eqnarray}\label{eq:index}
\big[N'_{\sigma'}: \mathbb{L}_{F}(N_{\sigma})]  &=&  \Bigg|\Z^{d}\Big / \prod_{i = 1}^{d} b_{i} \, \Z\Bigg| \notag\\
& = & \Bigg | \prod_{i = 1}^{d} \Z\big /  b_{i} \, \Z\Bigg| \notag\\
& = & b_{1}\cdots b_{d}. 
\end{eqnarray}\par
In the following we describe $[V(\sigma')]\in A_{r-n}(X_{\Sigma'})$: For this we denote by $e_i^{*} \in M := \Hom(N,\Z)$ and $e'^{*}_i \in M' := \Hom(N',\Z)$ the dual base vectors of $e_i$ and $e'_i$ and the dual map of $\mathbb{L}_{F}$ with
\begin{eqnarray}\label{eq:dual}
\varphi \colon M' &\to &M \\
m'& \mapsto &m' \circ \mathbb{L}_{F}. \notag
\end{eqnarray}
By \cite[Lem. 12.5.2]{cox} and using that $\Sigma'$ is smooth we can write $[V(\sigma')]$ as the intersection of Chow classes of hyperplanes
\begin{eqnarray}\label{eq:prod}
[V(\sigma')] =[ V(\rho'_{1})] \cdots [V(\rho'_{n})],
\end{eqnarray}
where $\rho'_i := \cone(e'_i)$ are pairwise distinct and $\sigma' = \sum_{i= 1}^{n} \rho'_{i}$. From the description of torus invariant divisors as in  \cite[Sec. 1]{fultonsturm} we get
\begin{eqnarray}\label{eq:div}
V(\rho'_{i}) = \di\Big(\chi^{{e'_{i}}^{*}}\Big),
\end{eqnarray}
where $\chi^{{e'_{i}}^{*}}\colon \T_{N'} \to \C^{\times}$ denotes the character to $ {e'_{i}}^{*} \in M'$. Since the pull back of a principal divisor is given by pulling back the defining equation (cf. \cite[Sec. 2.2]{fulton}), we conclude that the pull back
 of $[V(\rho'_{i})]$ is given by
\begin{eqnarray*}
f^{*}\big([V(\rho'_{i}) ]\big) \overset{\eqref{eq:prod}}{=}  f^{*}\Big(\Big[ \di\Big(\chi^{{e'_{i}}^{*}}\Big) \Big]\Big) = \Big[ \di\Big(\chi^{\varphi( {e'_{i}}^{*})}\Big) \Big].
\end{eqnarray*}
Since $\varphi( {e'_{i}}^{*})$ is in $M$, this gives a character $\chi^{{e'_{i}}^{*}}$ on $\T_{N}$ and we conclude
\begin{eqnarray}\label{gl:produkt}
 \Big[ \di\Big(\chi^{\varphi( {e'_{i}}^{*})}\Big) \Big] &=& \sum_{i =1, \ldots,r} \langle \varphi ({e'_{i}}^{*}), e_{i}\rangle \cdot [V (\rho_{i})]\notag\\
&\overset{\eqref{eq:dual}}{=} & \sum_{i =1, \ldots,r } \langle {e_{i}'}^{*} \circ \mathbb{L}_{F}, e_{i}\rangle \cdot[V (\rho_{i})] \notag\\
&= & b_{i} \cdot [V(\rho_{i})],
\end{eqnarray}
where $\rho_{i} := \cone(e_{i}) \in \Sigma$. Summing up we have 
$$ 
f^{*} \big([V(\rho'_{i}) ] \big) = b_{i} \cdot [V  (\rho_{i})]. 
$$
Finally we can compute $b_{\sigma}$ as:
\begin{eqnarray}
		f^{*}([V(\sigma')]) & \overset{\eqref{eq:prod}}{=}  &f^{*}\big([V(\rho'_{1})] \cdots [V(\rho'_{k})]\big) \label{eq:5}\\
& = &f^{*}\big([ V(\rho'_{1})] \big) \cdots f^{*}\big ([V(\rho'_{k})]\big) \label{eq:6}\\
&\overset{\eqref{gl:produkt}}{=} & b_{1}\cdot[ V(\rho_{1})] \cdot \ldots \cdot b_{d}\cdot [ V(\rho_{k})] \notag \\
& = & b_{1}\cdots b_{d}\cdot [V(\rho_{1})] \cdots [ V(\rho_{k})] \label{eq:7}\\
&= & b_{1}\cdots b_{d} \cdot[ V(\sigma)]\label{eq:8}\\
& \overset{\eqref{eq:index}}{=} & [N'_{\sigma'}:\mathbb{L}_{F}(N_{\sigma})] \cdot  [ V(\sigma)],\notag
\end{eqnarray}
where \eqref{eq:6} follows from \eqref{eq:5} with \cite[Prop. 2.3 (d)]{fulton} and \eqref{eq:8} from \eqref{eq:7} with \cite[Lem. 12.5.2]{cox}. 
And the pull back
 in \eqref{eq:pullback} given by
\begin{eqnarray*}
		f^{*}([V(\sigma')]) & = & \sum_{\substack{\sigma \in \Sigma_{n}:\\ F(\sigma) = \sigma'} } b_{\sigma}\cdot[V(\sigma)]\\
		 & = & \sum_{\substack{\sigma \in \Sigma_{n}:\\ F(\sigma) = \sigma'} } [N'_{\sigma'}: \mathbb{L}_{F}(N_{\sigma})] \cdot [V(\sigma)].
	\end{eqnarray*}
Plugging this in \eqref{eq:deg} we get
	\begin{eqnarray}
	 	\deg \big(c \smallfrown f^{*}([V(\sigma')])\big) &= & \deg \Big(c \smallfrown  \sum_{\substack{\sigma \in \Sigma_{n}:\\ F(\sigma) = \sigma'} }\big[N'_{\sigma'} : \mathbb{L}_{F}(N_{\sigma})\big] \cdot [V(\sigma)] \Big) \notag\\
		& =& \sum_{\substack{\sigma \in \Sigma_{n}: \notag\\ 
		F(\sigma) = \sigma'} } \deg \Big(c \smallfrown \big[N'_{\sigma'} : \mathbb{L}_{F}(N_{\sigma})\big] \cdot [V(\sigma)]\Big) \\
		& =& \sum_{\substack{\sigma \in \Sigma_{n}: \\ 
		F(\sigma) = \sigma'} }   \deg \big(c \smallfrown  [V(\sigma)]\big)  \cdot  \big[N'_{\sigma'} : \mathbb{L}_{F}(N_{\sigma})\big]  \label{prop:eq:1} \\
		& = & \sum_{\substack{\sigma \in \Sigma_{n}:\\ F(\sigma) = \sigma' } }m_{\sigma} \cdot\big[N'_{\sigma'} : \mathbb{L}_{F}(N_{\sigma})\big]  \label{prop:eq:2} \\
		& =: & m_{\sigma'} \notag,
\end{eqnarray}
where \eqref{prop:eq:2} follows from \eqref{prop:eq:1} with proposition \ref{prop:isomorphism}. 
\end{proof}

So in the case of tropical fans we deduce from \cite[Sec. 17.2]{fulton} that the formation of $F_{*}$ induces a functorial homomorphism of the tropical cycle groups. This gives another proof to the given one in \cite{gathman}.

\section{Pull back of tropical cycles}

\begin{definition}\label{def:pullback}
Let $F \colon N_{\R} \to N'_{\R}$ be an integral, $\R$-affine map. Then there exists a natural \emph{pull back
 morphism}
\begin{eqnarray*}
F^{*} \colon \TZ^k (N'_{\R}) & \to & \TZ^k (N_{\R})\\
C' & \mapsto & F^{*}C',
\end{eqnarray*}
which satisfies $|F^{*}C' | \subseteq F^{-1}(|C'|)$. It is given as follows: Let $C'$ in $\TZ^k (N'_{\R})$, we write $C' := (\mathscr{C'}, m')$ for a polyhedral complex $\mathscr{C'}$ of definition 
such that $|\mathscr{C'}| = N'_{\R}$. We also choose an integral, $\R$-affine polyhedral complex $\mathscr{C}$ with $|\mathscr{C}| = |N_{\R}|$. After a suitable refinement we may assume that we have $F(\sigma) \in \mathscr{C'}$ for all $\sigma \in \mathscr{C}$. We choose a generic vector $v'\in N'_{\R}$ and a sufficiently small $\varepsilon > 0$ (cf. \cite[Sec. 4]{fultonsturm}) and equip $\mathscr{C}$ with the weight function
\begin{eqnarray*}
m_{\tau} := \sum_{\substack{\sigma \in \mathscr{C}^{0}  : \\ \tau \prec \sigma }} \, \sum_{\substack{\tau' \in \mathscr{C'}^{k}  : \\ F(\tau) \prec \tau'}} m^{\tau}_{\sigma, \tau'} \cdot m'_{\tau'}
\end{eqnarray*} for every $\tau $ in $\mathscr{C}^k$. The coefficient $m^{\tau}_{\sigma, \tau'}$ is given by
	$$
	m^{\tau}_{\sigma, \tau'}  := \begin{cases}   [N' : \mathbb{L}_{F}(N_{\sigma})+ N'_{\tau'}] , & \text{ if } F(\sigma) \cap (\tau' + \varepsilon \cdot v') \not = \emptyset  \text{ and } \\
0, & \text{ otherwise.}\end{cases}
	$$
This turns $(\mathscr{C}^{\geq k},m)$ into a tropical cycle of codimension $k$. Now one defines $F^{*}C := (\mathscr{C}^{\geq k},m)$. The formation of $F^{*}$ is a functorial homomorphism of the tropical cycle groups.
\end{definition}

To make sure that this definition is indeed well-defined we reduce it to the case of tropical fans, where we can derive the statement from the the pull back
 of Minkowski weights.\\
Let $\tau \in \mathscr{C}^{k}$ and after a  choice of a generic vector (cf. definition \ref{def:intersectionproduct}) and a small $\varepsilon >0 $ we can assume that there is a $\tau' \in \mathscr{C'}^{k} $ such that $F(\tau) = \tau'$. The linear part of $F$ induces an integral linear map of the quotient vector spaces
$$
N_{\R}/ \mathbb{L}_{\tau}\to N'_{\R}/\mathbb{L}_{\tau'}.
$$
Locally we may assume (with the standard procedure) that $C$ and $C'$ are Minkowski weights in the complete fan $\Sigma$ in $N_{\R}/ \mathbb{L}_{\tau}$ and $\Sigma'$ in $N'_{\R}/\mathbb{L}_{\tau'}$. Futhermore we can assume that this map sends polyhedra to polyhedra which yields a toric morphism $f \colon X_{\Sigma} \to X_{\Sigma'} $ between the complete toric varieties $X_{\Sigma}$ and $X_{\Sigma'}$, and so induces a homomorphism of Chow cohomology groups 
$$ f^{*} \colon A^{k}(X_{\Sigma'}) \to A^{k}(X_{\Sigma}).$$
Since cocycles in the Chow cohomology groups correspond to Minkowski weights we see (since the lattice index does not change by passing to the quotient vector spaces) that the pull back
 of a Minkowski weight (cf. \cite[Cor. 3.7]{fultonsturm}) agrees with the pull back
 of tropical cycles given in definition \ref{def:pullback}.
Now the remaining properties follow from the properties of the Chow cohomology. \par

In \cite{fran} the author introduces a pull back
 using the intersection product for tropical cycle as in \cite{allermann} and shows that its definition coincide with the given one. We finish this section with an example which illustrate the definition.

\begin{example} 
If we assume in the situation of definition \ref{def:pullback} that the map $F \colon N_\R \to N'_\R $ is surjective then the computation of the pull back weight function simplifies to 
$$
	m_{\tau}  = \begin{cases}  m_{\tau'} \cdot [N' : \mathbb{L}_{F}(N)+ N'_{\tau'}] , & \text{ if codim}(\tau') = k   \text{ and } \\
0, & \text{ if codim}(\tau') < k .\end{cases}
	$$
Now for the projection map to the first coordinate $ \Z	^{2} \to \Z$ let
\begin{eqnarray*}
F \colon \R^{2} & \to & \R \\
(x,y) & \mapsto & x,
\end{eqnarray*}
 the induced map of vector spaces like in example \ref{ex:projection}. We will calculate the pull back $F^{*}C'$ of the tropical cycle $C'$ with polyhedral complex of definition $\mathscr{C}'$ given in figure \ref{fig:pullback} and with weight on all facets of $\mathscr{C}'$ equal to $1$. For this we choose a complete fan $\mathscr{C}$ in $N_\R$ pictured in figure \ref{fig:pullback}. With this choice the given polyhedral complexes of definition $\mathscr{C}$ and $\mathscr{C}'$ are compatible with respect to $F$. 
 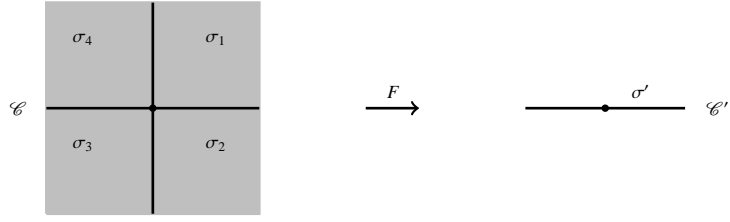
\begin{figure}[h]
  \centering

    \begin{tikzpicture}

     \begin{scope}[color = black ,text width=2em,scale=0.7,font={\scriptsize},line width=1pt]

   \filldraw[lightgray] (-2,-2) -- (2,-2) -- (2,2) -- (-2,2) -- (-2,-2);

    \draw[]  (0,0) -- (2,0);
    \draw[]  (0,0) -- (0,2);
    \draw[]  (0,0) -- (-2,0); 
    \draw[]  (0,0) -- (0,-2);

    \draw[]  (7,0) -- (10,0); 
    
    \fill (0,0) circle (0.07);

    \fill (8.5,0) circle (0.07);
    
        \draw[->]  (4,0) -- (5,0); 
    \node[above] at (4.9,0) {$F$};

         \node[above] at (1.5,1) {$\sigma_{1}$};
         \node[above] at (1.5,-1) {$\sigma_{2}$};
         \node[above] at (-1,-1) {$\sigma_{3}$};
         \node[above] at (-1,1) {$\sigma_{4}$};
         
     \node[left] at (-1.5,0) {$\mathscr{C}$};
    \node[above] at (9.5,0) {$\sigma'$};
    \node[right] at (10.2,0) {$\mathscr{C'}$};

    \end{scope}
    
   \end{tikzpicture} 
  \caption{Pull back of a tropical cycle.}
  \label{fig:pullback}
\end{figure}
Use the simplified formula to compute the weights on every facet we easily see that the weight for all facet is $1$. So $F^{*}C' = (\mathscr{C}^{\geq 0}, 1)$.

\end{example}

\section{Properties of tropical cycles}

In the last section we show that the intersection product commutes with the pull back and the projections formula holds.
\begin{proposition}\label{prop:pullbackschnitt}
Let $F \colon N_{\R} \to N'_{\R}$ be an integral, $\R$-affine map. For tropical cycles $C'_{1}$ and $C'_{2}$ in $N'_{\R}$ we have
$$
F^{*}(C'_{1}.C'_{2})= F^{*}(C'_{1}).F^{*}(C'_{2}).
$$

\end{proposition}

\begin{proof}
Let $k_i$ be the codimension of $C'_{i}$ and $\mathscr{C}_{i}$ a polyhedral complex of definition. After a suitable refinement we may assume that $\mathscr{C}'_{i}$ are subcomplexes in a complete polyhedral complex $\Sigma'$ and $\Sigma$ is a complete polyhedral complex in $N_{\R}$ such that  $F(\sigma) \in \Sigma'$ for all $\sigma \in \Sigma$.
Now we can assume with the definition of the stable intersection product and the pullback that the polyhedral complexes of definition for $F^{*}(C'_{1}.C'_{2})$ and $F^{*}(C'_{1}).F^{*}(C'_{2})$ agree. In the following we will denote it by $\mathscr{D}$. So it remains to compare the weights of every facet in $\mathscr{D}$.
We can do this locally. Let $\tau$ be a facet in $\mathscr{D}$. After choosing a generic vector we can assume, that $F(\tau) = \tau'$ has the same dimension in $\Sigma'$. Now the linear part of $F$ induces an integral linear map of the quotient vector spaces
$$
N_{\R}/ \mathbb{L}_{\tau}\to N'_{\R}/\mathbb{L}_{\tau'}.
$$
Now we can locally assume that all tropical cycles in question are tropical fans. Now the claim follows, with the standard procedure by identifying tropical fans with cocycles in the Chow cohomology groups, form the corresponding property of the Chow cohomology.
\end{proof}

\begin{proposition}\label{prop:projektionsformel}
Let $F \colon N_{\R} \to N'_{\R}$ be an integral, $\R$-affine map. For tropical cycles $C$ on $N_{\R}$ and $C'$ in $N'_{\R}$ we have
$$
F_{*}(F^{*}C'. C)= C'.F_{*}(C).
$$
\end{proposition}
\begin{proof}
Let $l$ and $(r'-k)$ be the dimension of $C$ and $C'$. At first we assume that all occurring polyheral complexes of definition for the tropical cycles are polyhedral subcomplexes in a complete and regular polyhedral complex $\Sigma$ in $N_{\R}$ and $\Sigma'$ in $N'_{\R}$. That means that $\Sigma_{\le l}$ and as well $\Sigma_{r-k}$ are a polyhedral complexes of definition for $C$ and $F^{*}C$. And that $\Sigma'_{\le r'-k}$ and $\Sigma'_{\le l}$ are polyhedral complexes of definition for $C'$ and $F_{*}C$. Furthermore we can assume that  $\tau \in F^{*} (\Sigma'_{\le r-k }). \Sigma_{\le l}$ and $\tau' \in \Sigma'_{\le r'-k}. F_{*} (\Sigma_{\le l})$ are facets with $F(\tau) = \tau'$ and 
$$
F^{\tau} \colon N_{\R}/ \mathbb{L}_{\tau}\to N'_{\R}/\mathbb{L}_{\tau'}
$$ is the integral linear map of the quotient vector spaces.
Locally in $\tau$ we get the tropical fans with rational fan of definition given by $\st_{\Sigma}(\tau)$ in $N_{\R}/ \mathbb{L}_{\tau}$ and $\st_{\Sigma}(\tau)$ in $N_{\R}/ \mathbb{L}_{\tau}$. We proof the claim while we are using that the projection formula already holds for tropical fans. This follows after identifying tropical fans with cocycles and using proposition \ref{prop:katz} from the properties of the Chow cohomology. Now the claim follows from the following computation:
\begin{eqnarray} 
\st_{\Sigma'_{\leq r'-k}. \Sigma'_{\leq l} }(\tau') &= & \st_{\Sigma'_{\leq r'-k}}(\tau') . \st_{\Sigma'_{\leq l} } (\tau') \label{eq:3}\\
& =  & \st_{\Sigma'_{\leq r'-k}}(\tau') . \bigg ( \sum_{\substack{\tau \in \Sigma_{l-k}: \\ F(\tau)= \tau'}} \big[N'_{\tau'}: \mathbb{L}_{\tau} (N_{\tau})\big] \, F^{\tau}_{*}\big(\st_{\Sigma_{\leq l}}(\tau)\big)  \bigg) \label{eq:4}\\
&  \overset{ }{=} & \sum_{\substack{\tau \in \Sigma_{l-k}: \\ F(\tau)= \tau'}} \big[N'_{\tau'}: \mathbb{L}_{\tau} (N_{\tau})\big] \,
  \bigg (\st_{\Sigma'_{\leq n'-k}}(\tau') . F^{\tau}_{*}\big(\st_{\Sigma_{\leq l}}(\tau)\big) \bigg ) \notag \\ 
& \overset{\text{}}{=}   & \sum_{\substack{\tau \in \Sigma_{l-k}: \\ F(\tau)= \tau'}}  \big[N'_{\tau'}:\mathbb{L}_{\tau}  (N_{\tau})\big] \,
\bigg (F^{\tau}_{*} \Big( (F^{\tau})^{*} \big(\st_{\Sigma'_{\leq r'-k }}(\tau') \big) . \st_{\Sigma_{\leq l}}(\tau)  \Big )\bigg )\notag\\
&  \overset{\text{}}{=}  & \sum_{\substack{\tau \in \Sigma_{l-k}: \\ F(\tau)= \tau'}} \big[N'_{\tau'}: \mathbb{L}_{\tau} (N_{\tau})\big]  \,
  \bigg (F^{\tau}_{*} \Big( \st_{F^{*} \big(\Sigma'_{\leq r'-k }\big)}(\tau)  . \st_{\Sigma_{\leq l}}(\tau)  \Big )\bigg )\notag\\
& = & \sum_{\substack{\tau \in \Sigma_{l-k}: \\ F(\tau)= \tau'}} \big[N'_{\tau'}: \mathbb{L}_{\tau}  (N_{\tau})\big] \,
 \Big (F^{\tau}_{*} \big( \st_{\Sigma_{\leq r-k }}(\tau) . \st_{\Sigma_{\leq l}}(\tau) \big)   \Big )\notag\\
& \overset{\text{}}{=} & \sum_{\substack{\tau \in \Sigma_{l-k}: \\ F(\tau)= \tau'}} \big[N'_{\tau'}: \mathbb{L}_{\tau}  (N_{\tau})\big] \, \Big (F^{\tau}_{*} \big( \st_{\Sigma_{\leq r-k } . \Sigma_{\leq l}}(\tau) \big) \Big )\notag\\
& \overset{\text{}}{=}  & \st_{F_{*} \big(\Sigma_{\leq r-k }. \Sigma_{\leq l}\big) }(\tau'),
\end{eqnarray}
where \eqref{eq:4} follows from \eqref{eq:3} with the local formula for the push forward given in \cite[Lem. 1.3.7]{rau}.
\end{proof}
Since the intersection product of \cite{allermann} coincides with our intersection product the statements of proposition \ref{prop:pullbackschnitt} und \ref{prop:projektionsformel} are also proved in \cite{fran}.

\newpage
\addcontentsline{toc}{chapter}{Literaturverzeichnis}
\printbibliography

Simon Flossmann, Fakultät für Mathematik, Universität Regensburg, Universitätsstraße 31, 93053 Regensburg, Germany, simon.flossmann@mathematik.uni-regensburg.de

\end{document}